\documentclass{amsart}
\usepackage{amsmath,amssymb,amsthm,url}

\usepackage[pagebackref,colorlinks=true]{hyperref}  

\usepackage{amsfonts}
\usepackage{amsmath,amsthm,amssymb,amscd,enumerate,eucal,url}

\numberwithin{equation}{section}
\newtheorem{thm}{Theorem}[section]

\newtheorem{prop}[thm]{Proposition}

\theoremstyle{definition}
\newtheorem{defn}{Definition}[section]
\newtheorem{rem}{Remark}[section]

\newcommand{\X}{\mathfrak{X}}
\newcommand{\s}{\mathfrak{S}}
\newcommand{\g}{\mathfrak{g}}
\newcommand{\W}{\mathcal{W}}
\newcommand{\R}{\mathbb{R}}

\newcommand{\norm}[1]{\left\Vert#1\right\Vert ^2}
\newcommand{\nJ}{\norm{\nabla J}}
\newcommand{\nN}{\norm{N}}
\newcommand{\ad}{{\rm ad}}
\newcommand{\tr}{{\rm tr}}
\def\co{\colon\thinspace}

\newcommand{\thmref}[1]{Theorem~\ref{#1}}

\newcommand{\tabref}[1]{Table~\ref{#1}}

\begin{document}

\title{On Lie groups as quasi-K\"ahler manifolds \\ with Killing Norden metric}

\author{Mancho Manev, Dimitar Mekerov}

\address{University of Plovdiv, Faculty of Mathematics and
Informatics\\
236 Bulgaria Blvd., 4003 Plovdiv, Bulgaria\\
\url{http://www.fmi-plovdiv.org/manev}}
\email{mmanev@uni-plovdiv.bg, mircho@uni-plovdiv.bg}

\begin{abstract}
A 6-parametric family of 6--dimensional quasi-K\"ahler manifolds
with Norden metric is constructed on a Lie group. This family is
characterized geometrically.
\end{abstract}



\date{2006/11/11}

\maketitle

\setcounter{tocdepth}{2} \tableofcontents

\section*{Introduction}

It is a fundamental fact that on an almost complex manifold with
Hermitian metric (almost Hermitian manifold), the action of the
almost complex structure on the tangent space at each point of the
manifold is isometry. There is another kind of metric, called a
Norden metric or a $B$-metric on an almost complex manifold, such
that the action of the almost complex structure is anti-isometry
with respect to the metric. Such a manifold is called an almost
complex manifold with Norden metric \cite{GaBo} or with $B$-metric
\cite{GaGrMi}. See also \cite{GrMeDj} for generalized
$B$-manifolds. It is known \cite{GaBo} that these manifolds are
classified into eight classes.

The purpose of the present paper is to exhibit, by construction,
almost complex structures with Norden metric on Lie groups as
6-manifolds, which are of a certain class, called the class of the
quasi-K\"ahler manifolds with Norden metrics.  This 6-parametric
family of manifolds is characterized geometrically.

The case of the lower dimension 4 is considered in \cite{GrMaMe-2}
and \cite{MeMaGr-3}.


\section{Almost Complex Manifolds with Norden Metric}\label{sec_1}

\subsection{Preliminaries}\label{sec-prelim}

Let $(M,J,g)$ be a $2n$-dimensional almost complex manifold with
Norden metric, i.~e. $J$ is an almost complex structure and $g$ is
a metric on $M$ such that
\begin{equation}\label{Jg}
J^2X=-X, \qquad g(JX,JY)=-g(X,Y)
\end{equation}
for all differentiable vector fields $X$, $Y$ on $M$, i.~e. $X, Y
\in \X(M)$.

The associated metric $\tilde{g}$ of $g$ on $M$ given by
$\tilde{g}(X,Y)=g(X,JY)$ for all $X, Y \in \X(M)$ is a Norden
metric, too. Both metrics are necessarily of signature $(n,n)$.
The manifold $(M,J,\tilde{g})$ is an almost complex manifold with
Norden metric, too.

Further, $X$, $Y$, $Z$, $U$ ($x$, $y$, $z$, $u$, respectively)
will stand for arbitrary differentiable vector fields on $M$
(vectors in $T_pM$, $p\in M$, respectively).

The Levi-Civita connection of $g$ is denoted by $\nabla$. The
tensor filed $F$ of type $(0,3)$ on $M$ is defined by
\begin{equation}\label{F}
F(X,Y,Z)=g\bigl( \left( \nabla_X J \right)Y,Z\bigr).
\end{equation}
It has the following symmetries
\begin{equation}\label{F-prop}
F(X,Y,Z)=F(X,Z,Y)=F(X,JY,JZ).
\end{equation}

Further, let $\{e_i\}$ ($i=1,2,\dots,2n$) be an arbitrary basis of
$T_pM$ at a point $p$ of $M$. The components of the inverse matrix
of $g$ are denoted by $g^{ij}$ with respect to the basis
$\{e_i\}$.

The Lie form $\theta$ associated with $F$ is defined by
\begin{equation}\label{theta}
\theta(z)=g^{ij}F(e_i,e_j,z).
\end{equation}

A classification of the considered manifolds with respect to $F$
is given in \cite{GaBo}. Eight classes of almost complex manifolds
with Norden metric are characterized there according to the
properties of $F$. The three basic classes are given as follows
\begin{equation}\label{class}
\begin{array}{l}
\W_1\co F(x,y,z)=\frac{1}{4n} \left\{
g(x,y)\theta(z)+g(x,z)\theta(y)\right. \\[4pt]
\phantom{\mathcal{W}_1\co F(x,y,z)=\frac{1}{4n} }\left.
    +g(x,J y)\theta(J z)
    +g(x,J z)\theta(J y)\right\};\\[4pt]
\W_2\co \mathop{\s} \limits_{x,y,z}
F(x,y,J z)=0,\quad \theta=0;\\[8pt]
\W_3\co \mathop{\s} \limits_{x,y,z} F(x,y,z)=0,
\end{array}
\end{equation}
where $\s $ is the cyclic sum by three arguments.

The special class $\W_0$ of the K\"ahler manifolds with Norden
metric belonging to any other class is determined by the condition
$F=0$.

\subsection{Curvature properties}\label{sec-curv}

Let $R$ be the curvature tensor field of $\nabla$ defined by
\begin{equation}\label{R}
    R(X,Y)Z=\nabla_X \nabla_Y Z - \nabla_Y \nabla_X Z -
    \nabla_{[X,Y]}Z.
\end{equation}
The corresponding tensor field of type $(0,4)$ is
determined as follows
\begin{equation}\label{R04}
    R(X,Y,Z,U)=g(R(X,Y)Z,U).
\end{equation}
The Ricci tensor $\rho$ and the scalar curvature $\tau$ are
defined as usual by
\begin{equation}\label{rho-tau}
    \rho(y,z)=g^{ij}R(e_i,y,z,e_j),\qquad \tau=g^{ij}\rho(e_i,e_j).
\end{equation}

Let $\alpha=\{x,y\}$ be a non-degenerate 2-plane (i.~e.
$\pi_1(x,y,y,x)=g(x,x)g(y,y) -g(x,y)^2 \neq 0$) spanned by vectors
$x, y \in T_pM, p\in M$. Then, it is known, the sectional
curvature of $\alpha$ is defined by the following equation
\begin{equation}\label{k}
    k(\alpha)=k(x,y)=\frac{R(x,y,y,x)}{\pi_1(x,y,y,x)}.
\end{equation}

The basic sectional curvatures in $T_pM$ with an almost complex
structure and a Norden metric $g$ are
\begin{itemize}
    \item \emph{holomorphic sectional curvatures} if $J\alpha=\alpha$;
    \item \emph{totally real sectional curvatures} if
    $J\alpha\perp\alpha$ with respect to $g$.
\end{itemize}

In \cite{GrDjMe}, a \emph{holomorphic bisectional curvature}
$h(x,y)$ for a pair of holomorphic 2-planes $\alpha_1=\{x,Jx\}$
and $\alpha_2=\{y,Jy\}$ is defined by
\begin{equation}\label{h}
    h(x,y)=-\frac{R(x,Jx,y,Jy)}
    {\sqrt{\pi_1(x,Jx,x,Jx)\pi_1(y,Jy,y,Jy)}},
\end{equation}
where $x$, $y$ do not lie along the totally isotropic directions,
i.~e. both of the couples $\bigl(g(x,x), g(x,Jx)\bigr)$ and
$\bigl(g(y,y), g(y,Jy)\bigr)$ are different from the couple
$\left(0,0\right)$. The holomorphic bisectional curvature is
invariant with respect to the basis of the 2-planes $\alpha_1$ and
$\alpha_2$. In particular, if $\alpha_1=\alpha_2$, then the
holomorphic bisectional curvature coincides with the holomorphic
sectional curvature of the 2-plane $\alpha_1=\alpha_2$.

\subsection{Isotropic K\"ahler manifolds}\label{sec-iK}
The square norm $\nJ$ of $\nabla J$ is defined in \cite{GRMa} by
\begin{equation}\label{snorm}
    \nJ=g^{ij}g^{kl}
    g\bigl(\left(\nabla_{e_i} J\right)e_k,\left(\nabla_{e_j}
    J\right)e_l\bigr).
\end{equation}

Having in mind the definition \eqref{F} of the tensor $F$ and the
properties \eqref{F-prop}, we obtain the following equation for
the square norm of $\nabla J$
\begin{equation}\label{snormF}
    \nJ=g^{ij}g^{kl}g^{pq}F_{ikp}F_{jlq},
\end{equation}
where $F_{ikp}=F(e_i,e_k,e_p)$.

\begin{defn}[\cite{MekMan-1}]\label{iK}
An almost complex manifold with Norden metric satisfying the
condition $\nJ=0$ is called an \emph{isotropic K\"ahler manifold
with Norden metric}.
\end{defn}

\begin{rem}
It is clear, if a manifold belongs to the class $\W_0$, then
it is isotropic K\"ahlerian but the inverse statement is not
always true.
\end{rem}

\section{Lie groups as Quasi-K\"ahler manifolds with \\ Killing Norden
metric}\label{sec-qK}

The only class of the three basic classes, where the almost
complex structure is not integrable, is the class $\W_3$ -- the
class of the \emph{quasi-K\"ahler manifolds with Norden metric}.

Let us remark that the definitional condition from \eqref{class}
implies the vanishing of the Lie form $\theta$ for the class
$\W_3$.

Let $V$ be a $2n$-dimensional vector space and consider the
structure of the Lie algebra defined by the brackets $
[E_i,E_j]=C_{ij}^kE_k, $ where $\{E_1,E_2,\dots,E_{2n}\}$ is a
basis of $V$ and $C_{ij}^k\in \R$.

Let $G$ be the associated connected Lie group and
$\{X_1,X_2,\dots,X_{2n}\}$ be a global basis of left invariant
vector fields induced by the basis of $V$. Then the Jacobi
identity has the form
\begin{equation}\label{Jac}
    \mathop{\s} \limits_{X_i,X_j,X_k}
    \bigl[[X_i,X_j],X_k\bigr]=0.
\end{equation}

Next we define an almost complex structure by the conditions
\begin{equation}\label{J}
JX_i=X_{n+i},\quad JX_{n+i}=-X_i,\qquad i\in\{1,2,\dots,n\}.
\end{equation}
Let us consider the left invariant metric defined by the following
way
\begin{equation}\label{g}
\begin{array}{l}
  g(X_i,X_i)=-g(X_{n+i},X_{n+i})=1, \qquad i\in\{1,2,\dots,n\} \\[4pt]
  g(X_j,X_k)=0,\qquad j\neq k \in \{1,2,\dots,2n\}. \\
\end{array}
\end{equation}
The introduced metric is a Norden metric because of \eqref{J}.

In this way, the induced $2n$-dimensional manifold $(G,J,g)$ is an
almost complex manifold with Norden metric, in short \emph{almost
Norden manifold}.

The condition the Norden metric $g$ be a Killing metric of the Lie
group $G$ with the corresponding Lie algebra $\g$ is $g(\ad
X(Y),Z)=-g(Y,\ad X(Z)$, where $X,Y,Z \in \g$ and $\ad X
(Y)=[X,Y]$. It is equivalent to the condition the metric $g$ to be
an invariant metric, i.~e.
\begin{equation}\label{inv}
    g\left([X,Y],Z\right)+g\left([X,Z],Y\right)=0.
\end{equation}

\begin{thm}\label{Th:Kil_g}
    If $(G,J,g)$ is an almost Norden manifold with a Killing metric $g$, then it is:
    \begin{enumerate}
    \renewcommand{\labelenumi}{(\roman{enumi})}
    \item
    a $\W_3$-manifold;
    \item
    a locally symmetric manifold.
    \end{enumerate}
\end{thm}
\begin{proof}
(i) Let $\nabla$ be the Levi-Civita connection of $g$. Then the
following well-known condition is valid
\begin{equation}\label{LC}
\begin{array}{l}
    2g(\nabla_X Y,Z)=Xg(Y,Z)+Yg(X,Z)-Zg(X,Y)\\[4pt]
    \phantom{2g(\nabla_X Y,Z)=}+g([X,Y],Z)+g([Z,X],Y)+g([Z,Y],X).
\end{array}
\end{equation}
The last equation and \eqref{inv} imply
\begin{equation}\label{invLC}
    \nabla_{X_i} X_j=\frac{1}{2}[X_i,X_j],\quad i,j\in\{1,2,\dots,2n\}.
\end{equation}

Then we receive consecutively
\[
\left( \nabla_{X_i} J \right)X_j=\nabla_{X_i} JX_j - J\nabla_{X_i}
X_j=\frac{1}{2}\bigl( [X_i, JX_j] - J[X_i, X_j] \bigr),
\]
\begin{equation}\label{F-inv}
F(X_i,X_j,X_k)=\frac{1}{2}\Bigl\{g\bigl( [X_i, JX_j],X_k\bigr) -
g\bigl([X_i, X_j],JX_k \bigr) \Bigr\}.
\end{equation}
According to \eqref{inv}, the last equation implies
$\mathop{\s}_{X_i,X_j,X_k} F(X_i,X_j,X_k)=0,$ i.~e. the manifold
belongs to the class $\W_3$.

(ii) In \cite{GrMaMe-2} is given the following form of the
curvature tensor
\[
    R(X_i,X_j,X_k,X_l)=-\frac{1}{4}g\Bigl(\bigl[[X_i,X_j],X_k\bigr],X_l]\Bigr).
\]
Using the condition \eqref{inv} for a Killing metric, we obtain
\begin{equation}\label{invRijks}
    R(X_i,X_j,X_k,X_l)=-\frac{1}{4}g\Bigl([X_i,X_j],[X_k,X_l]\Bigr).
\end{equation}

According to the constancy of the component $R_{ijks}$ and
\eqref{invLC} and \eqref{invRijks}, we get the covariant
derivative of the tensor $R$ of type $(0,4)$ as follows
\begin{equation}\label{nablaR}
\begin{split}
    \left( \nabla_{X_i} R
    \right)&(X_j,X_k,X_l,X_m)=\\
    =\frac{1}{8}
    &\Bigl\{
        g\Bigl(\bigl[[X_i,X_j],X_k\bigr]-\bigl[[X_i,X_k],X_j\bigr],[X_l,X_m]\bigr]\Bigr)\\[4pt]
        &+g\Bigl(\bigl[[X_i,X_l],X_m\bigr]-\bigl[[X_i,X_m],X_l\bigr],[X_j,X_k]\bigr]\Bigr)
    \Bigr\}.
\end{split}
\end{equation}
We apply the the Jacobi identity \eqref{Jac} to the double
commutators. Then the equation \eqref{nablaR} gets the form
\begin{equation}\label{nablaR=}
\begin{split}
    \left( \nabla_{X_i} R
    \right)(X_j,X_k,X_l,X_m)=-\frac{1}{8}
    &\Bigl\{
        g\Bigl(\bigl[X_i,[X_j,X_k]\bigr],[X_l,X_m]\Bigr)\\[4pt]
        &+g\Bigl(\bigl[X_i,[X_l,X_m]\bigr],[X_j,X_k]\bigr)\Bigr)
    \Bigr\}.
\end{split}
\end{equation}Since $g$ is a Killing metric, then applying \eqref{inv} to \eqref{nablaR=}
we obtain the identity $\nabla R=0$, i.~e. the manifold is locally
symmetric.
\end{proof}


\section{The Lie Group as a 6-Dimensional $\W_3$-Manifold}
\label{sec_6dim}

Let $(G,J,g)$ be a 6-dimensional almost Norden manifold with
Killing metric $g$. Having in mind \thmref{Th:Kil_g} we assert
that $(G,J,g)$ is a $\W_3$-manifold. Let the commutators have the
following decomposition
\begin{equation}\label{[]}
    [X_i,X_j]=\gamma_{ij}^k X_k,\quad \gamma_{ij}^k \in \R,\qquad
    i,j,k\in\{1,2,\dots,6\}.
\end{equation}

According to the condition \eqref{inv} for a Killing metric $g$,
the equations \eqref{[]} take the form given in \tabref{table1}.
\begin{table}
  \centering
  \caption{The Lie brackets with 20 parameters.}\label{table1}
\begin{tabular}{||c||c|c|c||c|c|c||}
  \hline
             & $X_1$ & $X_2$ & $X_3$ & $X_4$ & $X_5$ & $X_6$
             \\
  \hline\hline
  $[X_2,X_3]$ & $\lambda_{10}$ &                &                 & $ \lambda_7$ & $ \lambda_1$ & $ \lambda_2$ \\\hline
  $[X_3,X_1]$ &                & $\lambda_{10}$ &                 & $ \lambda_4$ & $ \lambda_8$ & $ \lambda_3$ \\\hline
  $[X_1,X_2]$ &                &                &  $\lambda_{10}$ & $ \lambda_5$ & $ \lambda_6$ & $ \lambda_9$ \\\hline
  \hline
  $[X_5,X_6]$ & $\lambda_{17}$ & $\lambda_{14}$ & $\lambda_{15}$ & $\lambda_{20}$ &                &                \\\hline
  $[X_6,X_4]$ & $\lambda_{11}$ & $\lambda_{18}$ & $\lambda_{16}$ &                & $\lambda_{20}$ &                \\\hline
  $[X_4,X_5]$ & $\lambda_{12}$ & $\lambda_{13}$ & $\lambda_{19}$ &                &                & $\lambda_{20}$ \\\hline
  \hline
  $[X_1,X_4]$ &              & $ \lambda_5$ & $-\lambda_4$ &                 & $-\lambda_{12}$ & $ \lambda_{11}$ \\\hline
  $[X_1,X_5]$ &              & $ \lambda_6$ & $-\lambda_8$ & $ \lambda_{12}$ &                 & $-\lambda_{17}$ \\\hline
  $[X_1,X_6]$ &              & $ \lambda_9$ & $-\lambda_3$ & $-\lambda_{11}$ & $ \lambda_{17}$ &                 \\\hline
  \hline
  $[X_2,X_4]$ & $-\lambda_5$ &              & $ \lambda_7$ &                 & $-\lambda_{13}$ & $ \lambda_{18}$ \\\hline
  $[X_2,X_5]$ & $-\lambda_6$ &              & $ \lambda_1$ & $ \lambda_{13}$ &                 & $-\lambda_{14}$ \\\hline
  $[X_2,X_6]$ & $-\lambda_9$ &              & $ \lambda_2$ & $-\lambda_{18}$ & $ \lambda_{14}$ &                 \\\hline
  \hline
  $[X_3,X_4]$ & $ \lambda_4$ & $-\lambda_7$ &              &                 & $-\lambda_{19}$ & $ \lambda_{16}$ \\\hline
  $[X_3,X_5]$ & $ \lambda_8$ & $-\lambda_1$ &              & $ \lambda_{19}$ &                 & $-\lambda_{15}$ \\\hline
  $[X_3,X_6]$ & $ \lambda_3$ & $-\lambda_2$ &              & $-\lambda_{16}$ & $ \lambda_{15}$ &                 \\\hline
\end{tabular}
\end{table}

Further we consider the special case when the condition
\begin{equation}\label{usl}
    g\Bigl([X_i,X_j],[X_k,X_l]\Bigr)=0
\end{equation}
for all mutually different indices $i,j,k,l$ in $\{1,2,\dots,6\}$
holds.

According to the Jacobi identity \eqref{Jac} and the condition
\eqref{usl}, \tabref{table1} is transformed to \tabref{table2}. 
\begin{table}
  \centering
  \caption{The Lie brackets with 6 parameters.}\label{table2}
\begin{tabular}{||c||c|c|c||c|c|c||}
  \hline
             & $X_1$ & $X_2$ & $X_3$ & $X_4$ & $X_5$ & $X_6$
             \\
  \hline\hline
  $[X_2,X_3]$ &              &              &              &              & $ \lambda_1$ & $ \lambda_2$ \\\hline
  $[X_3,X_1]$ &              &              &              & $ \lambda_4$ &              & $ \lambda_3$ \\\hline
  $[X_1,X_2]$ &              &              &              & $ \lambda_5$ & $ \lambda_6$ &              \\\hline
  \hline
  $[X_5,X_6]$ &              & $-\lambda_4$ & $-\lambda_5$ &              &              &              \\\hline
  $[X_6,X_4]$ & $-\lambda_1$ &              & $-\lambda_6$ &              &              &              \\\hline
  $[X_4,X_5]$ & $-\lambda_2$ & $-\lambda_3$ &              &              &              &              \\\hline
  \hline
  $[X_1,X_4]$ &              & $ \lambda_5$ & $-\lambda_4$ &              & $ \lambda_2$ & $-\lambda_1$ \\\hline
  $[X_1,X_5]$ &              & $ \lambda_6$ &              & $-\lambda_2$ &              &              \\\hline
  $[X_1,X_6]$ &              &              & $-\lambda_3$ & $ \lambda_1$ &              &              \\\hline
  \hline
  $[X_2,X_4]$ & $-\lambda_5$ &              &              &              & $ \lambda_3$ &              \\\hline
  $[X_2,X_5]$ & $-\lambda_6$ &              & $ \lambda_1$ & $-\lambda_3$ &              & $ \lambda_4$ \\\hline
  $[X_2,X_6]$ &              &              & $ \lambda_2$ &              & $-\lambda_4$ &              \\\hline
  \hline
  $[X_3,X_4]$ & $ \lambda_4$ &              &              &              &              & $-\lambda_6$ \\\hline
  $[X_3,X_5]$ &              & $-\lambda_1$ &              &              &              & $ \lambda_5$ \\\hline
  $[X_3,X_6]$ & $ \lambda_3$ & $-\lambda_2$ &              & $ \lambda_6$ & $-\lambda_5$ &              \\\hline
\end{tabular}
\end{table}

The Lie groups $G$ thus obtained are of a family which is
characterized by six real parameters $\lambda_i$ $(i = 1,\dots,
6)$. Therefore, for the manifold $(G,J,g)$ constructed above, we
establish the truthfulness of the following
\begin{thm}\label{Th:ex}
Let $(G,J,g)$ be a 6-dimensional almost Norden manifold, where $G$
is a connected Lie group with corresponding Lie algebra $\g$
determined by the global basis of left invariant vector fields
$\{X_1,X_2,\dots,X_6\}$; $J$ is an almost complex structure
defined by \eqref{J} and $g$ is an invariant Norden metric
determined by \eqref{g} and \eqref{inv}. Then $(G,J,g)$ is a
quasi-K\"ahler manifold with Norden metric if and only if $G$
belongs to the 6-parametric family of Lie groups determined by
\tabref{table2}.
\end{thm}

For the constructed manifold the Killing form $B(X,Y)=\tr(\ad
X,\ad Y)$ has the following determinant
\[
\det
B=64\left[(\lambda_1^2+\lambda_2^2+\lambda_3^2-\lambda_4^2-\lambda_5^2-\lambda_6^2)^2
-4(\lambda_1^2-\lambda_4^2)(\lambda_3^2-\lambda_6^2)\right]^3.
\]

A special case when the holomorphic sectional curvatures are zero
is considered in \cite{MaGrMe-4}. In the cited case we have
$\lambda_1=\lambda_4$, $\lambda_2=\lambda_5$,
$\lambda_3=\lambda_6$ and then the Killing form is degenerate.

Now, in the general case the question whether the Killing form $B$
can be a Norden metric is reasonable. The answer is negative
because the setting of the condition $B$ to be a Norden metric
implies $\lambda_i=0$ $(i=1,2,\dots,6)$.


\section{Geometric characteristics of the constructed manifold}

Let $(G,J,g)$ be the 6-dimensional $\W_3$-manifold introduced in
the previous section.

\subsection{The components of the tensor $F$}
Then by direct calculations, having in mind \eqref{F}, \eqref{J},
\eqref{g}, \eqref{inv}, \eqref{invLC}, \eqref{F-inv} and
\tabref{table2}, we obtain the nonzero components of the tensor
$F$ as follows

\begin{equation}\label{Fijk}
\begin{split}
\lambda_1&=2F_{116}=2F_{161}=-2F_{134}=-2F_{143}=2F_{223}=2F_{232}=2F_{256}\\[4pt]
\phantom{\lambda_1} &=2F_{265}=-F_{322}=-F_{355}=2F_{413}=2F_{431}=2F_{446}=2F_{464}\\[4pt]
\phantom{\lambda_1} &=-2F_{526}=-2F_{562}=2F_{535}=2F_{553}=-F_{611}=-F_{644},\\[4pt]
\lambda_2&=-2F_{115}=-2F_{151}=2F_{124}=2F_{142}=F_{233}=F_{266}=-2F_{323}\\[4pt]
\phantom{\lambda_2}&=-2F_{332}=-2F_{356}=-2F_{365}=-2F_{412}=-2F_{421}=-2F_{445}\\[4pt]
\phantom{\lambda_2}&=-2F_{454}=F_{511}=F_{544}=-2F_{626}=-2F_{662}=2F_{635}=2F_{653},\\[4pt]
\lambda_3&=-F_{133}=-F_{166}=-2F_{215}=-2F_{251}=2F_{224}=2F_{242}=2F_{313}\\[4pt]
\phantom{\lambda_3}&=2F_{331}=2F_{346}=2F_{364}=-F_{422}=-F_{455}=2F_{512}=2F_{521}\\[4pt]
\phantom{\lambda_3}&=2F_{545}=2F_{554}=2F_{616}=2F_{661}=-2F_{634}=-2F_{643},\\[4pt]
\lambda_4&=-2F_{113}=-2F_{131}=-2F_{146}=-2F_{164}=-2F_{226}=-2F_{262}\\[4pt]
\phantom{\lambda_4} &=2F_{235}=2F_{253}=F_{311}=F_{344}=2F_{416}=2F_{461}=-2F_{434}\\[4pt]
\phantom{\lambda_4} &=-2F_{523}=-2F_{443}=-2F_{532}=-2F_{556}=-2F_{565}=F_{622}=F_{655},\\[4pt]
\lambda_5&=2F_{112}=2F_{121}=2F_{145}=2F_{154}=-F_{211}=-F_{244}=-2F_{326}\\[4pt]
\phantom{\lambda_5} &=-2F_{362}=2F_{335}=2F_{353}=-2F_{415}=-2F_{451}=2F_{424}=2F_{442}\\[4pt]
\phantom{\lambda_5} &=-F_{533}=-F_{566}=2F_{623}=2F_{632}=2F_{656}=2F_{665},\\[4pt]
\lambda_6&=F_{122}=F_{155}=-2F_{212}=-2F_{221}=-2F_{245}=-2F_{254}=2F_{316}\\[4pt]
\phantom{\lambda_6} &=2F_{361}=-2F_{334}=-2F_{343}=F_{433}=F_{466}=-2F_{515}=-2F_{551}\\[4pt]
\phantom{\lambda_6} &=2F_{524}=2F_{542}=-2F_{613}=-2F_{631}=-2F_{646}=-2F_{664}.\\[4pt]
\end{split}
\end{equation}
where $F_{ijk}=F(X_i,X_j,X_k)$.

\subsection{The square norm of the Nijenhuis
tensor}

Let $N$ be the Nijenhuis tensor of the almost complex structure
$J$ on $G$, i.~e.
\begin{equation}\label{N}
    N(X,Y)=[X,Y]+J[JX,Y]+J[X,JY]-[JX,JY], \quad X, Y \in \g.
\end{equation}
Having in mind the equations in \tabref{table2}, we obtain the
nonzero components $N_{ij}=N(X_i,X_j)$ $(i, j=1,2,\dots,6)$ as
follows
\begin{equation}\label{Nij}
\begin{array}{l}
    N_{23}=-N_{56}=2\left(\lambda_4 X_2+\lambda_5 X_3+\lambda_1 X_5+\lambda_2 X_6\right),\\[4pt]
    N_{31}=-N_{64}=2\left(\lambda_1 X_1+\lambda_6 X_3+\lambda_4 X_4+\lambda_3 X_6\right),\\[4pt]
    N_{12}=-N_{45}=2\left(\lambda_2 X_1+\lambda_3 X_2+\lambda_5 X_4+\lambda_6 X_5\right),\\[4pt]
    N_{35}=-N_{26}=2\left(-\lambda_1 X_2-\lambda_2 X_3+\lambda_4 X_5+\lambda_5 X_6\right),\\[4pt]
    N_{16}=-N_{34}=2\left(-\lambda_4 X_1-\lambda_3 X_3+\lambda_1 X_4+\lambda_6 X_6\right),\\[4pt]
    N_{24}=-N_{15}=2\left(-\lambda_5 X_1-\lambda_6 X_2+\lambda_2 X_4+\lambda_3 X_5\right).\\[4pt]
\end{array}
\end{equation}
Therefore its square norm $\nN=g^{ik}g^{ks}g(N_{ij},N_{ks})$
vanishes, i.~e. $\nN=0$, where the nonzero components of the
inverse matrix of $g$ are the following
\begin{equation}\label{g^ij}
    g^{11}=g^{22}=g^{33}=-g^{44}=-g^{55}=-g^{66}=1.
\end{equation}
Then we have the following
\begin{prop}\label{Prop3.1}
    The Nijenhuis tensor of the manifold $(G,J,g)$ is isotropic.
\end{prop}

\subsection{The square norm of $\nabla J$}
According to \eqref{g},  \eqref{Fijk} and \eqref{g^ij},
from \eqref{snormF} we obtain that the square norm of $\nabla J$
is zero, i.~e. $\nJ=0$. Then we have the following
\begin{prop}\label{Prop:iK}
    The manifold $(G,J,g)$ is isotropic K\"ahlerian.
\end{prop}

\subsection{The components of $R$}
Let $R$ be the curvature tensor of type (0,4) determined by
\eqref{R04} and \eqref{R} on $(G,J,g)$. We denote its components
by $R_{ijks}=R(X_i,X_j,X_k,X_s)$; $i,j,k,s\in\{1,2,\dots,6\}$.
Using \eqref{invLC}, \eqref{Jac}, \eqref{invRijks} and
\tabref{table2} we get the nonzero components of $R$ as follows
\begin{equation}\label{Rijks}
\begin{array}{c}
\begin{array}{l}
    R_{1441}=-\frac{1}{4}\left(\lambda_1^2+\lambda_2^2-\lambda_4^2-\lambda_5^2\right),\\[4pt]
    R_{2552}=\frac{1}{4}\left(\lambda_1^2-\lambda_3^2-\lambda_4^2+\lambda_6^2\right),\\[4pt]
    R_{3663}=\frac{1}{4}\left(\lambda_2^2+\lambda_3^2-\lambda_5^2-\lambda_6^2\right),\\[4pt]
\end{array}
\\[4pt]
\begin{array}{ll}
    R_{1221}=-\frac{1}{4}\left(\lambda_5^2+\lambda_6^2\right),\quad
    &
    R_{1331}=-\frac{1}{4}\left(\lambda_3^2+\lambda_4^2\right),\\[4pt]
    R_{1551}=-\frac{1}{4}\left(\lambda_2^2-\lambda_6^2\right),\quad
    &
    R_{1661}=-\frac{1}{4}\left(\lambda_1^2-\lambda_3^2\right),\\[4pt]
    R_{2332}=-\frac{1}{4}\left(\lambda_1^2+\lambda_2^2\right),\quad
    &
    R_{2442}=-\frac{1}{4}\left(\lambda_3^2-\lambda_5^2\right),\\[4pt]
    R_{2662}=\frac{1}{4}\left(\lambda_2^2-\lambda_4^2\right),\quad
    &
    R_{3443}=\frac{1}{4}\left(\lambda_4^2-\lambda_6^2\right),\\[4pt]
    R_{3553}=\frac{1}{4}\left(\lambda_1^2-\lambda_5^2\right),\quad
    &
    R_{4554}=\frac{1}{4}\left(\lambda_2^2+\lambda_3^2\right),\\[4pt]
    R_{4664}=\frac{1}{4}\left(\lambda_1^2+\lambda_6^2\right),\quad
    &
    R_{5665}=\frac{1}{4}\left(\lambda_4^2+\lambda_5^2\right),\\[4pt]
\end{array}
\\[4pt]
\begin{array}{l}
    R_{1561}=R_{2562}=R_{3563}=-R_{4564}=\frac{1}{4}\lambda_1\lambda_2,\quad\\[4pt]
    R_{1341}=R_{2342}=-R_{5345}=-R_{6346}=-\frac{1}{4}\lambda_1\lambda_3,\\[4pt]
    R_{1361}=R_{2362}=-R_{4364}=-R_{5365}=\frac{1}{4}\lambda_1\lambda_4,\\[4pt]
    R_{1261}=R_{3263}=-R_{4264}=-R_{5265}=-\frac{1}{4}\lambda_1\lambda_5,\\[4pt]
    R_{2132}=-R_{4134}=-R_{5135}=-R_{6136}=\frac{1}{4}\lambda_1\lambda_6,\\[4pt]
    R_{3123}=-R_{4124}=-R_{5125}=-R_{6126}=\frac{1}{4}\lambda_2\lambda_3,\\[4pt]
    R_{1351}=R_{2352}=-R_{4354}=-R_{6356}=-\frac{1}{4}\lambda_2\lambda_4,\\[4pt]
    R_{1251}=R_{3253}=-R_{4254}=-R_{6256}=\frac{1}{4}\lambda_2\lambda_5,\\[4pt]
    R_{1241}=R_{3243}=-R_{5245}=-R_{6246}=-\frac{1}{4}\lambda_2\lambda_6,\\[4pt]
    R_{1461}=R_{2462}=R_{3463}=-R_{5465}=\frac{1}{4}\lambda_3\lambda_4,\\[4pt]
    R_{2152}=R_{3153}=-R_{4154}=-R_{6156}=-\frac{1}{4}\lambda_3\lambda_5,\\[4pt]
    R_{2142}=R_{3143}=-R_{5145}=-R_{6146}=\frac{1}{4}\lambda_3\lambda_6,\\[4pt]
    R_{1231}=-R_{4234}=-R_{5235}=-R_{6236}=\frac{1}{4}\lambda_4\lambda_5,\\[4pt]
    R_{2162}=R_{3163}=-R_{4164}=-R_{5165}=-\frac{1}{4}\lambda_4\lambda_6,\\[4pt]
    R_{1451}=R_{2452}=R_{3453}=-R_{6456}=\frac{1}{4}\lambda_5\lambda_6.\\[4pt]
\end{array}
\end{array}
\end{equation}

\subsection{The components of $\rho$ and the value of $\tau$}
Having in mind \eqref{rho-tau}, \eqref{g^ij} and \eqref{Rijks}, we
obtain the components $\rho_{ij}=\rho(X_i,X_j)$
$(i,j=1,2,\dots,6)$ of the Ricci tensor $\rho$ and the the value
of the scalar curvature $\tau$ as follows
\begin{equation}\label{rho_ij}
\begin{array}{c}
\begin{array}{l}
    \rho_{11}=\frac{1}{2}\left(\lambda_1^2+\lambda_2^2-\lambda_3^2
                                -\lambda_4^2-\lambda_5^2-\lambda_6^2\right),\\[4pt]
    \rho_{22}=\frac{1}{2}\left(-\lambda_1^2-\lambda_2^2+\lambda_3^2
                                +\lambda_4^2-\lambda_5^2-\lambda_6^2\right),\\[4pt]
    \rho_{33}=\frac{1}{2}\left(-\lambda_1^2-\lambda_2^2-\lambda_3^2
                                -\lambda_4^2+\lambda_5^2+\lambda_6^2\right),\\[4pt]
    \rho_{44}=\frac{1}{2}\left(-\lambda_1^2-\lambda_2^2-\lambda_3^2
                                +\lambda_4^2+\lambda_5^2-\lambda_6^2\right),\\[4pt]
    \rho_{55}=\frac{1}{2}\left(\lambda_1^2-\lambda_2^2-\lambda_3^2
                                -\lambda_4^2-\lambda_5^2+\lambda_6^2\right),\\[4pt]
    \rho_{66}=\frac{1}{2}\left(-\lambda_1^2+\lambda_2^2+\lambda_3^2
                                -\lambda_4^2-\lambda_5^2-\lambda_6^2\right),\\[4pt]
\end{array}
\\[4pt]
\begin{array}{lll}
    \rho_{12}=\lambda_2\lambda_3,\qquad
    &
    \rho_{13}=\lambda_1\lambda_6,\qquad
    &
    \rho_{14}=\lambda_3\lambda_6,\\[4pt]
    \rho_{15}=-\lambda_3\lambda_5,\quad
    &
    \rho_{16}=-\lambda_4\lambda_6,\quad
    &
    \rho_{23}=\lambda_4\lambda_5;\\[4pt]
    \rho_{24}=-\lambda_2\lambda_6,\quad
    &
    \rho_{25}=\lambda_2\lambda_5,\quad
    &
    \rho_{26}=-\lambda_1\lambda_5;\\[4pt]
    \rho_{34}=-\lambda_1\lambda_3,\quad
    &
    \rho_{35}=-\lambda_2\lambda_4,\quad
    &
    \rho_{36}=\lambda_1\lambda_4;\\[4pt]
    \rho_{45}=\lambda_5\lambda_6,\quad
    &
    \rho_{46}=\lambda_3\lambda_4,\quad
    &
    \rho_{56}=\lambda_1\lambda_2;\\[4pt]
\end{array}
\end{array}
\end{equation}
\begin{equation}\label{tau}
    \tau=0,
\end{equation}
i.~e. the scalar curvature on $(G,J,g)$ is zero.

The last equation implies immediately
\begin{prop}\label{Prop:tau=0}
    The manifold $(G,J,g)$ is scalar flat.
\end{prop}

\subsection{The sectional curvatures and the holomorphic bisectional curvature}

Let us consider the characteristic 2-planes $\alpha_{ij}$ spanned
by the basis vectors $\{X_i,X_j\}$ at an arbitrary point of the
manifold:
\begin{itemize}
    \item holomorphic 2-planes: $\alpha_{14}$, $\alpha_{25}$, $\alpha_{36}$;
    \item pairs of totally real 2-planes:
    $\left(\alpha_{12}, \alpha_{45}\right)$; $\left(\alpha_{13}, \alpha_{46}\right)$;
    $\left(\alpha_{15}, \alpha_{24}\right)$; $\left(\alpha_{16}, \alpha_{34}\right)$;
    $\left(\alpha_{23}, \alpha_{56}\right)$; $\left(\alpha_{26},
    \alpha_{35}\right)$.
\end{itemize}
Then, using \eqref{k}, \eqref{g} and \eqref{Rijks}, we obtain the
corresponding sectional curvatures
\begin{equation}\label{k_ik}
\begin{array}{c}
\begin{array}{l}
k(\alpha_{14})=\frac{1}{4}\left(\lambda_1^2+\lambda_2^2-\lambda_4^2-\lambda_5^2\right),\\[4pt]
k(\alpha_{25})=-\frac{1}{4}\left(\lambda_1^2-\lambda_3^2-\lambda_4^2+\lambda_6^2\right),\\[4pt]
k(\alpha_{36})=-\frac{1}{4}\left(\lambda_2^2+\lambda_3^2-\lambda_5^2-\lambda_6^2\right);\\[4pt]
\end{array}
\; \\[4pt]
    \begin{array}{ll}
    k(\alpha_{12})=-\frac{1}{4}\left(\lambda_5^2+\lambda_6^2\right),\quad
    &
    k(\alpha_{45})=\frac{1}{4}\left(\lambda_2^2+\lambda_3^2\right),\\[4pt]
    k(\alpha_{13})=-\frac{1}{4}\left(\lambda_3^2+\lambda_4^2\right),\quad
    &
    k(\alpha_{46})=\frac{1}{4}\left(\lambda_1^2+\lambda_6^2\right),\\[4pt]
    k(\alpha_{15})=\frac{1}{4}\left(\lambda_2^2-\lambda_6^2\right),\quad
    &
    k(\alpha_{24})=\frac{1}{4}\left(\lambda_3^2-\lambda_5^2\right),\\[4pt]
    k(\alpha_{16})=\frac{1}{4}\left(\lambda_1^2-\lambda_3^2\right),\quad
    &
    k(\alpha_{34})=-\frac{1}{4}\left(\lambda_4^2-\lambda_6^2\right),\\[4pt]
    k(\alpha_{23})=-\frac{1}{4}\left(\lambda_1^2+\lambda_2^2\right),\quad
    &
    k(\alpha_{56})=\frac{1}{4}\left(\lambda_4^2+\lambda_5^2\right),\\[4pt]
    k(\alpha_{26})=-\frac{1}{4}\left(\lambda_2^2-\lambda_4^2\right),\quad
    &
    k(\alpha_{35})=-\frac{1}{4}\left(\lambda_1^2-\lambda_5^2\right).\\[4pt]
\end{array}
\end{array}
\end{equation}

Taking into account \eqref{h}, \eqref{g} and \eqref{Rijks}, we
obtain that the holomorphic bisectional curvature of the three
pairs of the basis holomorphic 2-planes vanishes, i.~e.
\begin{equation}\label{h=0}
    h(\alpha_{14},\alpha_{25})=h(\alpha_{14},\alpha_{36})=h(\alpha_{25},\alpha_{36})=0.
\end{equation}

\subsection{The isotropic-K\"ahlerian property} Having in mind Propositions \ref{Prop3.1}--\ref{Prop:iK} and \thmref{Th:Kil_g},
we give the following characteristics of the constructed manifold
\begin{thm}
    The manifold $(G,J,g)$ constructed as an $\W_3$-manifold with Killing metric in \thmref{Th:ex}:
    \begin{enumerate}
    \renewcommand{\labelenumi}{(\roman{enumi})}
    \item
    is isotropic K\"ahlerian;
    \item
    has an isotropic Nijenhuis tensor;
    \item
    is scalar flat;
    \item
    is locally symmetric.
    \end{enumerate}
\end{thm}


\end{document}